\documentclass[11pt,amssymb]{amsart}
\usepackage{amssymb}
\usepackage[mathscr]{eucal}
\usepackage[all,cmtip]{xy}

\newcommand{\im}{{\rm im}\:}

\newcommand{\Z}{{\mathbb Z}}

\newcommand{\Q}{{\mathbb Q}}

\newcommand{\Br}{\mathrm{Br}}

\newcommand{\Ga}{\mathrm{Gal}}
\newtheorem{thm}{Theorem}[section]
\newtheorem{lemma}[thm]{Lemma}
\newtheorem{prop}[thm]{Proposition}
\newtheorem{cor}[thm]{Corollary}

\newcommand{\he}{H_{\mathrm{\acute{e}t}}}

\font\brus=wncyr10.240pk scaled 1200 .240pk
\DeclareFontFamily{U}{wncy}{}
    \DeclareFontShape{U}{wncy}{m}{n}{<->wncyr10}{}
    \DeclareSymbolFont{mcy}{U}{wncy}{m}{n}
    \DeclareMathSymbol{\Sha}{\mathord}{mcy}{"58}
\begin{document}
\title[Finiteness results over higher-dimensional fields]{Some finiteness results for algebraic groups and unramified cohomology over higher-dimensional fields}

\author[A.~Rapinchuk]{Andrei S. Rapinchuk}
\author[I.~Rapinchuk]{Igor A. Rapinchuk}

\address{Department of Mathematics, University of Virginia,
Charlottesville, VA 22904-4137, USA}

\email{asr3x@virginia.edu}

\address{Department of Mathematics, Michigan State University, East Lansing, MI
48824, USA}

\email{rapinchu@msu.edu}


\begin{abstract}
We formulate and analyze several finiteness conjectures for linear algebraic groups over higher-dimensional fields. In fact, we prove all of these conjectures for algebraic tori as well as in some other situations. This work relies in an essential way on several finiteness results for unramified cohomology.
\end{abstract}

\maketitle

\section{Introduction}\label{S-Introduction}

The recent works \cite{CRR4}, \cite{CRR-Spinor}, and \cite{CRR-Israel} have brought to the forefront several related (conjectural) finiteness properties of linear algebraic groups over an arbitrary finitely generated field $K$. Recall that any such field is equipped with an almost canonical set $V$ of discrete valuations called {\it divisorial}. More precisely, $V$ consists of the discrete valuations of $K$ associated with the prime divisors of {\it model} ${X}$, i.e. a normal scheme of finite type over $\Z$ having $K$ as its function field\footnotemark \footnotetext{Note that any two divisorial sets $V_1$ and $V_2$ associated with models of $K$ are commensurable, i.e. $V_i \setminus (V_1 \cap V_2)$ is finite for $i = 1,2$, and for any finite subset $S$ of a divisorial set $V$, the set $V \setminus S$ contains a divisorial set.}. For the formulation of our first conjecture, we will say that a reductive algebraic $K$-group $G$ has {\it good reduction} at a place $v$ of $K$ if there exists a reductive group scheme $\mathscr{G}$ over the valuation ring $\mathcal{O}_v$ of the completion $K_v$ whose generic fiber $\mathscr{G}  \times_{\mathcal{O}_v} K_v$ is isomorphic to $G \times_K K_v$ (see \cite{Con} or \cite{DemGr} for an exposition of the theory of reductive group schemes).

\vskip2mm

\noindent {\bf Conjecture 1.} {\it Let $G_0$ be a (connected) reductive algebraic group over a finitely generated field $K$, and $V$ be a divisorial set of places of $K$. Then the set of $K$-isomorphism classes of (inner) $K$-forms $G$ of $G_0$ that have good reduction at all $v \in V$ is finite (at least when the characteristic of $K$ is ``good").}

\vskip2mm

\noindent (When $G$ is an absolutely almost simple algebraic group, we say that ${\rm char}~K = p$ is ``good" for $G$ if either $p = 0$ or $p>0$ and does not divide the order of the Weyl group of $G$. For non-semisimple reductive groups, only characteristic 0 will be considered good.)

\vskip2mm

This is one of the central conjectures in the rapidly-evolving study of algebraic groups over higher-dimensional fields, and it has implications for a number of topics of current interest, including the genus problem for absolutely almost simple algebraic groups and the analysis of weakly commensurable Zariski-dense subgroups of these groups (see \cite[\S1]{CRR-Spinor} for a detailed discussion).  As observed in \cite{CRR2}, Conjecture 1 also yields the truth of the following conjecture for {\it semisimple adjoint} groups.

\vskip2mm

\noindent {\bf Conjecture 2.} {\it Let $G$ be a (connected) reductive algebraic group defined over a finitely generated field $K$, and $V$ be a divisorial set of places of $K$. Then the global-to-local map in Galois cohomology $$\lambda_{G , V} \colon H^1(K , G) \longrightarrow \prod_{v \in V} H^1(K_v , G)$$ is \emph{proper}, i.e. the pre-image of a finite set is finite. In particular, the Tate-Shafarevich set $$\text{\brus SH}(G , V) := \ker \lambda_{G , V}$$ is finite.}

\vskip2mm

The properness of $\lambda_{G,V}$ in the classical setting where $K$ is a number field is well-known (see \cite[Ch. III, \S4.6]{Serre-GC}) and is established using reduction theory (see also \cite{ConradFiniteness} for the function field case). Yet another famous consequence of reduction theory is the finiteness of the class number. We briefly recall the relevant definitions here and refer the reader to \S\ref{S-CondT} for further details. So, suppose $K$ is a field endowed with a set $V$ of discrete valuations, and let $G$ be an algebraic $K$-group with a fixed matrix realization $G \subset \mathrm{GL}_n.$ For each $v \in V$, we set $G(\mathcal{O}_v) = G(K_v) \cap \mathrm{GL}_n (\mathcal{O}_v)$ and then define the corresponding {\it adelic group} as
$$
G(\mathbb{A}(K , V)) = \{\, (g_v) \in \prod_{v \in V} G(K_v) \ \vert \ g_v \in G(\mathcal{O}_v) \ \ \text{for almost all} \ \ v \in V \}
$$
(where, as above, for each discrete valuation $v$, we let $K_v$ denote the completion of $K$ at $v$ and $\mathcal{O}_v \subset K_v$ the corresponding valuation ring).
The product
$$
G(\mathbb{A}^{\infty}(K , V)) = \prod_{v \in V} G(\mathcal{O}_v)
$$
is called the {\it subgroup of integral adeles}. Furthermore, assume that $V$ satisfies the following condition (which holds automatically for a divisorial
set of places of a finitely generated field):

\vskip2mm

\noindent (A) \ For any  $a \in K^{\times}$, the set $V(a) := \{ v \in V \, \vert \, v(a) \neq 0 \}$ is finite.

\vskip2mm

\noindent Then there is a diagonal embedding $G(K) \hookrightarrow G(\mathbb{A}(K , V))$, whose image is called the {\it subgroup of principal adeles} and which we will still denote simply by $G(K).$ The set of double cosets
$$
\mathrm{cl}(G, K, V) := G(\mathbb{A}^{\infty}(K , V)) \backslash G(\mathbb{A}(K , V)) / G(K)
$$
is called the {\it class set} of $G$ (we should point out that the class set is sometimes defined using {\it rational adeles} rather than the full adeles, as we have done here).

Note that if $G = \mathbb{G}_m$ is a 1-dimensional split torus, then there is a bijection between $\mathrm{cl}(G, K, V)$ and the Picard group $\mathrm{Pic}(V)$ (defined as the quotient of the free abelian group on $V$ by the subgroup of `principal divisors,' which makes sense in view of condition (A) --- see \cite[\S2]{CRR-Spinor}); when $K$ is a number field and $V$ is the set of all nonarchimedean places, this simply becomes the usual class group of $K$. Moreover, if $G = \mathrm{O}_n(q)$ is the orthogonal group of a nondegenerate $n$-dimensional quadratic form $q$ over a number field $K$ and $V$ is again the set of all nonarchimedean places of $K$, then the elements of $\mathrm{cl}(G, K, V)$ are in bijection with the classes in the genus of $q$ --- see, for example, \cite[Proposition 8.4]{Pl-R}. (These two examples explain the terminology.) It was proved by Borel (\cite[\S5]{Borel}) that if $K$ is a number field and $V$ is the set of nonarchimedean places of $K$, then the class set $\mathrm{cl}(G, K, V)$ is finite for any linear algebraic group $G$ over $K$, which generalizes the classical results about the finiteness of the class number of a number field and the number of classes in the genus of a quadratic form. More recently, Borel's finiteness theorem was extended to all algebraic groups over global fields of positive characteristic by B.~Conrad \cite{ConradFiniteness} using the theory of pseudo-reductive groups developed by Conrad-Gabber-Prasad \cite{CGP}. On the other hand, given an arbitrary finitely generated field $K$ equipped with a set $V$ of discrete valuations, the Picard group $\mathrm{Pic}(V)$ may very well be infinite, which raises the question of how the finiteness theorem for the class number of an algebraic group over a number field can be extended to more general fields. In \cite{CRR-Israel}, we proposed to consider the following

\vskip2mm

\noindent {\bf Condition (T).} {\it There exists a finite subset $S \subset V$ such that $\vert \mathrm{cl}(G, K, V \setminus S) \vert = 1.$}

\vskip2mm

\noindent It is easy to see that if the class set $\mathrm{cl}(G, K, V)$ is finite, then Condition (T) holds for the given $G$, $K$, and $V$. In general,
as pointed out in \cite[\S3]{CRR-Israel}, Condition (T) is instrumental in the study of the finiteness properties of unramified cohomology, which have a close connection to Conjectures 1 and 2. While one does not expect Condition (T) to hold for an arbitrary reductive algebraic group $G$ over a general finitely generated field $K$ and a divisorial set $V$, it is likely to be true for all $G$ in certain important situations, including when

\vskip1mm

\noindent $\bullet$ \parbox[t]{16.5cm}{$K$ is a {\it 2-dimensional global field} (i.e. the function field of a smooth geometrically integral curve over a number field or the function field of a smooth geometrically integral surface over a finite field --- see \cite{Kato} and \cite{CRR-Spinor}) and $V$ is a divisorial set of places; and}

\vskip1mm

\noindent $\bullet$ \parbox[t]{16.5cm}{$K = k(C)$, the function field of a smooth geometrically integral curve $C$ over a finitely generated field $k$ and $V$ is the set of places of $K$ associated with the closed points of $C$.}

\vskip2mm

Our first goal in this paper is to establish Conjectures 1 and 2 for all algebraic tori over a finitely generated field $K$ of characteristic 0 with respect to any divisorial set $V$ of discrete valuations. Furthermore, while the analogue of Condition (T) for rational adeles was previously established for algebraic tori over finitely generated fields in \cite[Proposition 4.2]{CRR-Israel}, we will verify it here for the full adeles and also prove a more precise result that yields Condition (T) for (disconnected) groups whose connected component is a torus. As an application, one then obtains Condition (T) for the normalizer of a maximal torus in a reductive group. We formulate our main results pertaining to Conjectures 1 and 2 below, and refer the reader to \S\ref{S-CondT} (particularly Proposition \ref{P-ToriT1} and Theorem \ref{T-ToriT}) for the detailed statements concerning Condition (T).

\begin{thm}\label{T-ToriGoodReduction}
Let $K$ be a finitely generated field of characteristic 0 and $V$ be a divisorial set of places of $K$. Then for any integer $d \geq 1$, the set of $K$-isomorphism classes of $d$-dimensional $K$-tori that have good reduction at all places $v \in V$ is finite.
\end{thm}

\begin{thm}\label{T-FinSha}
Let $K$ be a finitely generated field of and $V$ be a divisorial set of places of $K$. Then for any algebraic $K$-torus $T$, the Tate-Shafarevich group
$$
\Sha^1(T,V) = \ker \left(H^1 (K,T) \to \prod_{v \in V} H^1(K_v, T) \right)
$$
is finite.
\end{thm}

A salient feature of the paper is the systematic use of adelic techniques, particularly in \S\S \ref{S-CondT}-\ref{S-FinSha}. While these have long
been indispensable tools in the analysis of global fields, their arithmetic applications in the context of more general fields have been scarce. We develop some basic results
for the adele groups of algebraic tori in \S\ref{S-CondT} and use these to establish Condition (T) in certain cases. We then apply them in \S\ref{S-FinSha} to give one of two proofs of Theorem \ref{T-FinSha}.
Our second proof of Theorem \ref{T-FinSha} in \S\ref{S-FinSha}, which requires certain restrictions
on $\mathrm{char}\: K$, relies on some considerations that have been used to establish the finiteness of unramified cohomology. It should be pointed out that the connections between Conjectures 1 and 2 and the finiteness properties of unramified cohomology exist not only in the case of tori but also in other important
situations --- \cite{CRR-Spinor}. So, we conclude the paper by obtaining in \S\ref{S-UnramCoh} several new finiteness results for unramified cohomology and discussing their applications.

To be more precise, we first briefly recall the basic set-up. Let $K$ be a field equipped with a discrete valuation $v$ and having residue field $K^{(v)}$, and suppose that $M$ is a finite Galois module that is unramified at $v$ and whose order is prime to ${\rm char}~K^{(v)}.$ Then, for all $i \geq 1$, there exist residue maps
$$
\partial^i_{v, M} \colon H^i(K, M) \to H^{i-1}(K^{(v)}, M(-1))
$$
(see \cite[Ch. II, \S7]{GMS} for the details). An element of $H^i (K,M)$ is said to be {\it unramified at $v$} if it lies in $\ker \partial^i_{v,M}.$ Moreover, if $V$ is a set of discrete valuations of $K$ such that the residue maps exists for all $v \in V$ (cf. Condition (B) in \S\ref{S-UnramCoh}), one defines the {\it degree $i$ \it unramified cohomology of $K$ with respect to $V$} by
$$
H^i (K,M)_V = \bigcap_{v \in V} \ker \partial^i_{v,M}.
$$
Unramified cohomology initially emerged in the study of rationality questions for algebraic varieties, but has since become an important tool in a variety of problems involving algebraic cycles, algebraic groups, division algebras, quadratic forms, etc. (see, for example, \cite{CT-SB} for a detailed discussion of unramified cohomology and its applications). In connection with Conjectures 1 and 2, our focus in this paper is on the finiteness properties of unramified cohomology of finitely generated fields. In \S\ref{S-UnramCoh}, we obtain finiteness statements for the unramified cohomology of function fields of rational surfaces and Severi-Brauer varieties over number fields (Theorem \ref{T-PurelyTran} and Proposition \ref{P-SB}) and derive consequences for Conjectures 1 and 2 for several classes of groups, including spinor and special orthogonal groups of quadratic forms and groups of type $\mathsf{G}_2$ (Theorems \ref{T-Thm1} and \ref{T-Thm2}).


\section{Tori with good reduction: Proof of Theorem \ref{T-ToriGoodReduction}}

In this section, we will prove Theorem \ref{T-ToriGoodReduction} concerning algebraic tori with good reduction. Throughout this section, we will work over fields of characteristic 0; our main assertion is in fact {\it false} in positive characteristic --- see Remark 2.5 below.

An important ingredient needed for our argument is a higher-dimensional version of the Hermite-Minkowski theorem, formulated in Proposition \ref{P-HigherHM} below. Recall that the classical version of the theorem states that given a number field $L$, a finite set $T$ of primes of $L$, and an integer $n \geq 1$, there are only finitely many extensions $L'/L$ of degree $n$ that are unramified outside $T$. This fact can also be interpreted group-theoretically as follows. Following Serre, we say that a profinite group $G$ is {\it of type} (F) if for every integer $n$, $G$ has only a finite number of open subgroups of index $n$ --- cf. \cite[Ch. III, \S4.1]{Serre-GC} (profinite groups of type (F) are also sometimes called {\it small}). By Galois theory, the Hermite-Minkowski theorem translates into the statement that the Galois group $G_{F,T}$ of the maximal Galois extension of $F$ unramified outside of $T$ is of type (F).

Suppose now that $S$ is a regular integral scheme that is of finite type and dominant over ${\rm Spec}(\Z)$. Denote by $K$ the function field of $S$ (thus, in particular, ${\rm char}~K = 0$), fix an algebraic closure $\overline{K}$of $K$, and let $\overline{s} \colon {\rm Spec}(\overline{K}) \to S$ be the corresponding geometric point of $S$. Let $V$ be the set of discrete valuations of $K$ associated with the codimension 1 points of $S$ and set $K_V/K$ to be the maximal subextension of $\overline{K}$ that is unramified at all $v \in V$. With these notations, we have

\begin{prop}\label{P-HigherHM}
The extension $K_V/K$ is Galois and $\Ga(K_V/K)$ is of type {\rm (F)}.
\end{prop}

We begin the proof with the following

\begin{lemma}\label{L-ZNPurity}
With the above notations, let $K_S/K$ be the compositum of all finite subextensions $L/K$ of $\overline{K}$ such that the normalization of $S$ in $L$ is \'etale over $S$. Then $K_S = K_V.$
\end{lemma}
\begin{proof}
It follows from the definitions that we have the inclusion $K_S \subset K_V.$ To show the reverse inclusion, suppose that $L/K$ is a finite subextension of $\overline{K}$ that is unramified at all $v \in V$, and let $Y$ be the normalization of $S$ in $L.$ Then by assumption, $Y \to S$ is finite \'etale over each codimension 1 point of $S$. The Zariski-Nagata purity theorem, whose statement we include below for completeness, then implies that $Y$ is \'etale over $S$, hence $L \subset K_S.$
\end{proof}

\begin{thm}\label{T-ZariskiNagata}{\rm (Zariski-Nagata purity theorem)} Let $\varphi \colon Y \to S$ be a finite surjective morphism of integral schemes, with $Y$ normal and $S$ regular. Assume that the fiber of $Y_P$ of $\varphi$ above each codimension 1 point of $S$ is \'etale over the residue field $\kappa(P).$ Then $\varphi$ is \'etale.

\end{thm}

\noindent (See, for example, \cite[Theorem 5.2.13]{Szamuely} for the statement and related discussion and \cite[Exp. X, Th\'eor\`eme 3.4]{SGA2} for a detailed proof.)

\vskip2mm

Next, it is well-known that under our assumptions, $K_S/K$ is a Galois extension, and $\Ga(K_S/K)$ is canonically isomorphic to the fundamental group $\pi_1^{\text{\'et}}(S, \bar{s})$, for the geometric point $\bar{s} \colon {\rm Spec}(\overline{K}) \to S$ (see, for example, \cite[Proposition 5.4.9]{Szamuely}). Thus, Proposition \ref{P-HigherHM} follows from Lemma \ref{L-ZNPurity} and the following

\begin{thm}\label{T-Small}$($cf. \cite[Theorem 2.9]{HH} and \cite[Ch. VI, \S2.4]{Faltings}$)$ Let $X$ be a connected scheme of finite type and dominant over ${\rm Spec}(\Z).$ Then the fundamental group $\pi_1^{\text{\'et}}(X)$ (with respect to any geometric point) is of type {\rm (F)}.
\end{thm}

\noindent We now turn to

\vskip1mm

\noindent {\it Proof of Theorem \ref{T-ToriGoodReduction}.} Let $K$ be a finitely generated field of characteristic 0 and $X$ be a model of $K$. After possibly replacing $X$ by an open subset, we may assume that $X$ is smooth over an open subset of ${\rm Spec}(\Z)$ (see, for example, \cite[Proposition 6.16 and Corollary 14.34]{GW}); in particular, the structure morphism is flat, hence open, and therefore $X$ is dominant over ${\rm Spec}(\Z)$. Let $V$ be the divisorial set of discrete valuations of $K$ associated with the prime divisors of $X$. We will work with a fixed algebraic closure $\overline{K}$ of $K$.

Recall that the $K$-isomorphism classes of $d$-dimensional $K$-tori are in one-to-one correspondence with the equivalence classes of continuous representations $\rho \colon \Ga (\overline{K}/K) \to {\rm GL}_d(\Z)$ (see, for example, \cite[\S 2.2.4]{Pl-R}). Moreover, it is well-known that a $K$-torus has good reduction at a place $v$ of $K$ if and only if $T \times_K K_v$ splits over an unramified extension of the completion $K_v$; in particular, this means that the inertia subgroup $I_v$ of the decomposition group $D_v \subset \Ga (\overline{K}/K)$ acts trivially on the character group $X(T)$ of $T$ and hence $I_v \subset \ker \rho$ (see, for example, \cite[1.1]{NX}) Thus, using the preceding notations, the $K$-isomorphism classes of $d$-dimensional $K$-tori having good reduction at $v \in V$ are in bijection with the equivalence classes of continuous representations $\rho \colon \Ga(K_V/K) \to {\rm GL}_d(\Z).$

Next, since by Minkowski's Lemma the congruence subgroup $\mathrm{GL}_d(\mathbb{Z} , 3)$ modulo 3, i.e. the kernel of the reduction modulo 3 homomorphism $\mathrm{GL}_d(\mathbb{Z}) \to \mathrm{GL}_d(\mathbb{Z}/3\mathbb{Z})$, is torsion-free (see, for example, \cite[Lemma 4.19]{Pl-R}), it follows that the minimal splitting fields of $d$-dimensional tori have {\it bounded} degree over $K$.
Consequently, Proposition \ref{P-HigherHM} implies that there are only finitely many possibilities for the minimal splitting fields of $d$-dimensional $K$-tori that have good reduction at all $v \in V.$ Finally, by reduction theory, $\mathrm{GL}_d(\Z)$ has only finitely many conjugacy classes of finite subgroups (see \cite[Theorem 4.3]{Pl-R}), so for any such splitting field $L/K$, there are only finitely many equivalence of representations $\rho \colon \Ga(L/K) \to {\rm GL}_d(\Z).$ The assertion of Theorem \ref{T-ToriGoodReduction} now follows. $\Box$

\vskip2mm

\noindent {\bf Remark 2.5.} It should be pointed out that Theorem \ref{T-ToriGoodReduction}
is generally false in characteristic $p > 0$. The element of the above argument that fails in this case is the Hermite-Minkowski Theorem. Indeed, let $K = k(t)$ be the field of rational functions in one variable over the prime field $k = \mathbb{F}_p$, and let $V$ be the set of discrete valuations of $K$ corresponding to all monic irreducible polynomials $f(t) \in k[t]$ (in other words, $V$ consists of the valuations corresponding to all closed points of $\mathbb{A}^1_k = \mathbb{P}^1_k \setminus \{ \infty \}$). Let $\wp(a) = a^p - a$ be the Artin-Schreier operator. Then for any $a \in k[t] \setminus \wp(k[t])$, the Artin-Schreier polynomial $f_a(t) = t^p - t - a$ defines a degree $p$ cyclic Galois extension $L_a/K$. Since $f'_a = -1$, this extension is unramified at all $v \in V$. On the other hand, the quotient $k[t] / \wp(k[t])$ is easily seen to be infinite, so we have infinitely many degree $p$ cyclic extensions of $K$ unramified at all $v \in V$. For any such extension $L/K$, the corresponding quasi-split torus $T = \mathrm{R}_{L/K}(\mathbb{G}_m)$ and the norm torus $T' = \mathrm{R}^{(1)}_{L/K}(\mathbb{G}_m)$ have good reduction at all $v \in V$, implying that there are infinitely many isomorphism classes of such tori.

\vskip2mm

\noindent {\bf Remark 2.6.} Let $G$ be an absolutely almost simple algebraic group over a field $K$, and let $L$ be the minimal Galois extension of $K$ over which $G$ becomes an inner form of a split group. It is well known that the Galois group $\mathrm{Gal}(L/K)$ is isomorphic to a subgroup of the automorphism group of the Dynkin diagram of $G$, hence $L/K$ can only be of degree $1, 2, 3$ or $6$. Furthermore, if $G$ has good reduction at a discrete valuation $v$ of $K$, then $v$ must be unramified in the corresponding extension $L/K$. So, if $K$ is a finitely generated field of characteristic zero with a divisorial set of places $V$, then Proposition \ref{P-HigherHM} implies that there exists a finite collection $L_1, \ldots , L_r$ of Galois extensions of $K$ of degree $1, 2, 3$ or $6$ such that for any absolutely almost simple $K$-group $G$ having good reduction at all $v \in V$, the corresponding extension $L$ is one of the $L_i$'s. In fact, this remains valid also when $K$ has characteristic $p > 0$ different from $2$ and $3$. It follows  in these situations that if the finiteness of the set $K$-isomorphism classes in Conjecture 1 is known to hold only for {\it inner} $K$-form $G$ of $G_0$ with good reduction at all $v \in V$ and all quasi-split groups $G_0$ associated with one of the $L_i$'s, then it actually holds for {\it all} forms. On the other hand, if $p = 2$ or $3$, then,
as we have seen in Remark 2.5, the field $K = \mathbb{F}_p(t)$ has infinitely many cyclic degree $p$ extensions $L/K$ unramified at all $v \in V$, where, as above, $V$ is the set of all places of $K$ different from $\infty.$
The quasi-split simply connected $K$-groups $G^L$ associated with these extensions will then constitute an infinite family of pairwise nonisomorphic $K$-groups having good reduction at all $v \in V$. Thus, particular care needs to be taken when trying to extend Conjecture 1 to characteristics $2$ and $3$.

\vskip2mm

\noindent {\bf Remark 2.7.} It is not difficult to see that if a reductive group $G$ over a field $K$ has good reduction at a discrete valuation $v$ of $K$, then so does the maximal central torus $T$ of $G$. So, for $K$ a finitely generated field of characteristic zero with a divisorial set of place $V$, we conclude from Theorem \ref{T-ToriGoodReduction} that for a given reductive $K$-group $G$, there exists a \emph{finite} collection $T_1, \ldots , T_r$ of algebraic $K$-tori such that if $G'$ is a $K$-form of $G$ that has good reduction at all $v \in V$, then the maximal central torus $T'$ of $G'$ is $K$-isomorphic to one of the $T_i$'s. This essentially reduces the proof of Conjecture 1 to semisimple groups.

\section{Condition (T) for (disconnected) algebraic groups with toroidal connected component}\label{S-CondT}

Our goal in this section is two-fold. First, we verify Condition (T) for a torus over a finitely generated field with respect to a divisorial set of places --- note that this was done in \cite[Proposition 4.2]{CRR-Israel} for {\it rational} adeles, and in Proposition \ref{P-ToriT1} below,
we establish the corresponding fact for the full adeles defined in \S\ref{S-Introduction}. We then prove Theorem \ref{T-ToriT}, which yields Condition (T) for algebraic groups over finitely generated fields whose connected component is a torus. In fact, we develop, more generally, a strategy for verifying Condition (T) for a disconnected linear algebraic group given that it holds for the group's connected component.


For the reader's convenience, we begin by briefly reviewing our notations. Throughout this section, we take $K$ to be a finitely generated field and $V$ a divisorial set of places of $K$. Note that $V$ satisfies condition (A) that was introduced in \S\ref{S-Introduction}. Recall that we denote by $\mathbb{A}(K , V)$ the corresponding $K$-algebra of adeles, i.e. the restricted (topological) product of the completions $K_v$ for $v \in V$ with respect to the valuation rings $\mathcal{O}_v \subset K_v$ (cf. \cite[Ch. II, \S\S13-14]{ANT}, where the construction is described in detail for global fields). Furthermore, we let
$$
\mathbb{A}^{\infty}(K,V) = \prod_{v \in V} \mathcal{O}_v
$$
denote the subring of $\mathbb{A}(K,V)$ of integral adeles. Next, suppose $L/K$ is a finite separable field extension of $K$, and let $V^L$ denote the set of all extensions to $L$ of the discrete valuations in $V$. Note that $V^L$ is a divisorial set of places of $L$: indeed, it consists of the discrete valuations corresponding to the prime divisors on the normalization in $L$ of the chosen model for $K$.
It is well-known that there exists a natural isomorphism
$$
\mathbb{A}(K,V) \otimes_K L \simeq \mathbb{A}(L, V^L)
$$
of topological rings (cf. \cite[Ch. II, \S14]{ANT}). In particular, for a finite Galois extension $L/K$, the adele ring
$\mathbb{A}(L , V^L)$ has a natural action of the Galois group $\Ga(L/K)$ such that
$$
\mathbb{A}(L , V^L)^{\mathrm{Gal}(L/K)} = \mathbb{A}(K , V).
$$

Now let $G$ be a linear algebraic $K$-group with a fixed matrix realization $G \subset \mathrm{GL}_n.$ Then the group of points
$$
G(\mathbb{A}(K,V)) := \left( \prod_{v \in V} G(K_v) \right) \cap \mathrm{GL}_n(\mathbb{A}(K,V))
$$
is naturally identified with the adele group of $G$ introduced in \S\ref{S-Introduction}. The subgroup of integral adeles is given by
$$
G(\mathbb{A}^{\infty}(K , V)) = G(\mathbb{A}(K , V)) \cap \mathrm{GL}_n(\mathbb{A}^{\infty}(K , V)) = \prod_{v \in V} G(\mathcal{O}_v).
$$
Since $V$ satisfies condition (A),  we have the diagonal embedding $K \hookrightarrow \mathbb{A}(K , V)$ that yields the diagonal embedding $G(K) \hookrightarrow G(\mathbb{A}(K , V))$, the image of which is called the group of {\it principal adeles} and is routinely identified with $G(K)$. As in \S\ref{S-Introduction}, the set of double cosets $G(\mathbb{A}^{\infty}(K , V)) \backslash G(\mathbb{A}(K , V)) / G(K)$ is called the {\it class set} and denoted $\mathrm{cl}(G, K, V)$. We say that Condition (T) holds in this situation if there exists a finite subset $S \subset V$ such that $\mathrm{cl}(G, K, V \setminus S)$ reduces to a single element.

As mentioned above, our first main result in this section is the following.

\begin{prop}\label{P-ToriT1}
Condition {\rm (T)} holds for any algebraic torus $T$ over a finitely generated field $K$ with respect to any divisorial set
of places $V$ of $K$.
\end{prop}

This statement is a straightforward consequence of the more general Proposition \ref{P-ToriT2} below. To formulate the result, we will need one additional bit of notation. For each $v \in V$, denote by $\Delta_v$ the (unique) maximal bounded subgroup of $T(K_v)$. Note that if $T$ has good reduction at $v$, then $\Delta_v = T(\mathcal{O}_v)$ (see, for example, \cite{Marco}). On the other hand, since $V$ satisfies (A), the torus $T$ {\it does} have good reduction at almost all $v \in V$. Consequently, we have $\Delta_v = T(\mathcal{O}_v)$ for all but finitely many $v \in V$, and therefore
$$
\Delta := \prod_{v \in V} \Delta_v
$$
naturally embeds into $T(\mathbb{A}(K,V)).$

\begin{prop}\label{P-ToriT2}
The group
$$
T(\mathbb{A}(K,V))/(\Delta \cdot T(K))
$$
is finitely generated.
\end{prop}

In the proof of this proposition, as well as in later arguments, we will need the following well-known fact from the cohomology of finite groups (see, for example, \cite[Ch. II, Corollary 1.32]{MilneCFT} for a proof).

\begin{lemma}\label{L-FiniteCoh}
Let $G$ be a finite group and suppose $M$ is a $G$-module that is finitely generated as an abelian group. Then the cohomology groups $H^i(G,M)$ are finite for all $i \geq 1.$
\end{lemma}

\vskip2mm

\noindent {\it Proof of Proposition \ref{P-ToriT2}.} Let $L = K_T$ be the minimal Galois extension of $K$ over which $T$ splits, and denote by $V^L$ the set of all extensions to $L$ of places in $V$. For each $w \in V^L$, let $\tilde{\Delta}_w$ denote the (unique) maximal bounded subgroup of $T(L_w).$ Then for $v \in V$ and $w \vert v$, we have the inclusion $\Delta_v \subset \tilde{\Delta}_w$, and hence a diagonal embedding
$$
\Delta_v \hookrightarrow \prod_{w \vert v} \tilde{\Delta}_w =: \tilde{\Delta}(v).
$$
Notice that $\tilde{\Delta}(v)$ is the maximal bounded subgroup of
$$
\prod_{w \vert v} T(L_w) = T(L \otimes_K K_v).
$$
Moreover, it is invariant under the natural action of $\Ga(L/K)$ and the subgroup of Galois-fixed elements is a bounded subgroup of $T(K_v)$, hence coincides with $\Delta_v.$ Set
$$
\tilde{\Delta} = \prod_{w \in V^L} \tilde{\Delta}_w = \prod_{v \in V} \tilde{\Delta}(v).
$$
Then $\tilde{\Delta}$ embeds into $T(\mathbb{A}(L , V^L)) = T(\mathbb{A}(K , V) \otimes_K L)$. It is clearly invariant under the action of $\Ga(L/K)$ and the subgroup of Galois-fixed elements coincides with $\Delta$.

Now, by the definition of $L$, there exists an $L$-isomorphism $T \simeq (\mathbb{G}_m)^d$ where $d = \dim T$. It induces an isomorphism between $T(\mathbb{A}(L , V^L))$ and $(\mathbb{I}(L , V^L))^d$, where $\mathbb{I}(L , V^L) := \mathbb{G}_m(\mathbb{A}(L , V^L))$ is the group of ideles of $L$ with respect to $V^L$ (see \cite[\S2]{CRR-Spinor} for a detailed discussion of ideles in this context).
Under this isomorphism, $T(L)$ maps to $(L^{\times})^d$, and $\tilde{\Delta}$ to $(\mathbb{I}^{\infty}(L , V^L))^d$, where
$$
\mathbb{I}^{\infty}(L , V^L) = \prod_{w \in V^L} \mathcal{O}^{\times}_{L_w}
$$
is the subgroup of integral ideles. This yields an isomorphism
$$
T(\mathbb{A}(L , V^L))/(\tilde{\Delta} \cdot T(L)) \simeq [\mathbb{I}(L , V^L) / \mathbb{I}^{\infty}(L , V^L) L^{\times}]^d.
$$
On the other hand, by \cite[Lemma 2.2]{CRR-Spinor}, the quotient $\mathbb{I}(L , V^L) / \mathbb{I}^{\infty}(L , V^L) L^{\times}$ can be identified with the Picard group $\mathrm{Pic}(V^L)$ for $V^L$, and, as we observed in \cite[\S4]{CRR-Spinor}, it follows from \cite{Kahn1} that $\mathrm{Pic}(V^L)$ is finitely generated. So, we conclude that the group $T(\mathbb{A}(L , V^L))/(\tilde{\Delta} \cdot T(L))$ is also finitely generated. We also note that the intersection $T(L) \cap \tilde{\Delta}$ is mapped to $E^d$, where $$E = \{ x \in L^{\times} \ \vert \ v(x) = 0 \ \ \text{for all} \ \ v \in V^L \}$$ is the group of units in $L^{\times}$ with respect to $V^L$. The main result of \cite{Samuel} implies that $E$ is a finitely generated group (see, for example, \cite[\S8]{CRR3} for the details), so the intersection $T(L) \cap \tilde{\Delta}$ is also finitely generated.

Next, set
$$
\Omega = T(\mathbb{A}(K,V)) \bigcap (\tilde{\Delta}\cdot T(L)).
$$
By construction, the quotient $T(\mathbb{A}(K,V))/\Omega$ embeds into $T(\mathbb{A}(L , V^L))/\tilde{\Delta} \cdot T(L)$ and is therefore finitely generated. Thus, to complete the proof, it is enough to show that
$$
\Omega/(\Delta \cdot T(K))
$$
is finite. For this, we note that there is an injective homomorphism
$$
\lambda \colon \Omega/(\Delta \cdot T(K)) \hookrightarrow H^1(L/K, U), \footnotemark
$$
\footnotetext{Following the usual notations, we use $H^1 (L/K, U)$ to denote the Galois cohomology group $H^1(\Ga (L/K), U)$.}where $U = T(L) \cap \tilde{\Delta}.$ The construction of $\lambda$ is essentially identical to that of an analogous map on rational adeles given in the proof of \cite[Proposition 4.2]{CRR-Israel}, so we only sketch the argument here. Given $t \in \Omega$, we write it as $t = t_1t_2$, with $t_1 \in \tilde{\Delta}$ and $t_2 \in T(L).$ Considering the natural action of $\mathcal{G} = \Ga(L/K)$ on $T(\mathbb{A}(L, V^L))$ that leaves $\tilde{\Delta}$ invariant, for every $\sigma \in \mathcal{G}$, we have
$$
\xi(\sigma) := \sigma(t_1)^{-1} t_1 = \sigma(t_2)t_2^{-1} \in \tilde{\Delta} \cap T(L) = U.
$$
It is easy to see that $\{\xi(\sigma)\}_{\sigma \in \mathcal{G}}$ defines a 1-cocyles on $\mathcal{G}$ with values in $U$. Moreover, the class of $\xi$ in $H^1(\mathcal{G}, U) = H^1(L/K, U)$ depends only on $t$, but not on the choice of a particular factorization $t = t_1t_2$, and thus will be denoted by $\xi_t.$ In this way, we obtain a homomorphism
$$
\tilde{\lambda} \colon \Omega \to H^1(L/K, U), \ \ \ t \mapsto \xi_t,
$$
which then induces the required map $\lambda.$ Since, as noted above, $U$ is finitely generated, it follows from Lemma \ref{L-FiniteCoh} that  $H^1(L/K, U)$ is finite, and hence
$$
\Omega/(\Delta \cdot T(K))
$$
is finite, as required. $\Box$

\vskip3mm

For the sake of completeness, let us now briefly indicate how Proposition \ref{P-ToriT1} follows from Proposition \ref{P-ToriT2}. As we already mentioned above, there exists a finite subset $S_1 \subset V$ such that $\Delta_v = T(\mathcal{O}_v)$ for all $v \in V \setminus S_1.$ By Proposition \ref{P-ToriT2}, we can find $t_1, \dots, t_r \in T(\mathbb{A}(K,V))$ so that
their images in the quotient $T(\mathbb{A}(K , V))/ (\Delta \cdot T(K))$ generate the latter. Thus, there exists a finite subset $S_2 \subset V$ such that the $v$-component $(t_i)_v$ belongs to $T(\mathcal{O}_v)$ for all $v \in V \setminus S_2$ and all $i = 1, \ldots , r$. Set $S = S_1 \cup S_2$, and let $\pi \colon T(\mathbb{A}(K , V)) \to T(\mathbb{A}(K , V \setminus S))$ be the natural projection. By our construction, $\pi(\Delta) \subset T(\mathbb{A}^{\infty}(K , V \setminus S))$ and $\pi(t_i) \in T(\mathbb{A}^{\infty}(K , V \setminus S))$ for all $i = 1, \ldots , r$. Since $\Delta$ and $T(K)$, together with the elements $t_1, \ldots , t_r$ generate $T(\mathbb{A}(K , V))$ as an abstract group, we obtain
$$
T(\mathbb{A}(K , V \setminus S)) = \pi(T(\mathbb{A}(K , V))) = T(\mathbb{A}^{\infty}(K , V \setminus S)) \cdot T(K),
$$
as required.

\vskip3mm

Next, we will extend Proposition \ref{P-ToriT1} to disconnected groups whose connected component is a torus. This, in particular, applies to the normalizers of maximal tori in reductive groups (see Corollary \ref{C-ToriT3} below), which is likely to be helpful for the analysis of finiteness properties in the general case.


\begin{thm}\label{T-ToriT}
Let $K$ be a finitely generated field and $V$ be a divisorial set of places of $K$. Then any linear algebraic $K$-group $G$ whose connected component is a torus satisfies Condition {\rm (T)}.
\end{thm}

The proof of the theorem will rely on Lemmas \ref{L-1} and \ref{L-2} below, which in fact provide a general approach for establishing Condition (T) for a linear algebraic $K$-group once it is known to hold for the group's connected component (cf. Remark 3.8).
Let us note that although we only consider finitely generated fields and divisorial sets of places, the lemmas are actually valid for any field $K$ equipped with a set $V$ of discrete valuations
satisfying condition (A) introduced in \S\ref{S-Introduction}.

\begin{lemma}\label{L-1}
Let $G$ be a linear algebraic $K$-group with connected component $H = G^{\circ}.$
Assume that
\begin{equation}\label{E:E1}
G(K_v) = G(\mathcal{O}_v) H(K_v)
\end{equation}
for almost all $v \in V$. If $H$ satisfies Condition {\rm (T)}, then so does $G$.
\end{lemma}
\begin{proof}
Deleting finitely many places from $V$, we may assume that (\ref{E:E1}) holds for all $v \in V$. Then for any subset $V' \subset V$, the natural map
$$
H(\mathbb{A}^{\infty}(K , V')) \backslash H(\mathbb{A}(K , V')) / H(K) \longrightarrow G(\mathbb{A}^{\infty}(K , V')) \backslash G(\mathbb{A}(K , V')) / G(K)
$$
is surjective. By our assumption, the term on the left reduces to a single element for some cofinite $V' \subset V$. But then so does the term on the right (with the same $V'$), i.e. $G$ satisfies Condition (T).
\end{proof}

Our next statement will be instrumental for verifying condition (\ref{E:E1}). We continue to denote by $H$ the connected component of a linear algebraic $K$-group $G$. It follows from \cite[AG 13.3]{Borel} that $G(\overline{K}) = G(K^{\rm sep}) H(\overline{K})$, where $\overline{K}$ denotes a fixed algebraic closure of $K$ and $K^{\rm sep} \subset \overline{K}$ the corresponding separable closure. Consequently, one can find
a finite Galois extension $L/K$ and a finite subset $C \subset G(L)$ such that $G(\overline{K}) = C H(\overline{K})$. Then $G(F) = C H(F)$ for every field extension $F/L.$

\begin{lemma}\label{L-2}
With notations as above, fix a place $v \in V$ and let $w$ be an extension of $v$ to $L$. Assume that

\vskip1mm

\noindent {\rm (i)} $C \subset G(\mathcal{O}_{L_w})$;

\vskip1mm

\noindent {\rm (ii)} the natural map $H^1(L_w/K_v , H(\mathcal{O}_{L_w})) \to H^1(L_w/K_v , H(L_w))$ has trivial kernel.

\vskip1mm

\noindent Then {\rm (\ref{E:E1})} holds.
\end{lemma}
\begin{proof}
Let $g \in G(K_v)$. By construction, we can write $g = ch$ with $c \in C$ and $h \in H(L_w)$. Then for any $\sigma \in \Ga(L_w/K_v)$, we have $\sigma(g) = g$, hence
$$
c^{-1}\sigma(c) = h \sigma(h)^{-1} =: \xi(\sigma).
$$
Clearly, $\xi = \{\xi(\sigma)\}_{\sigma \in \Ga(L_w/K_v)}$ is a 1-cocycle on $\Ga(L_w/K_v)$ with values in
$$
G(\mathcal{O}_{L_w}) \cap H(L_w) = H(\mathcal{O}_{L_w}),
$$
and the corresponding cohomology class lies in the kernel of the map $$H^1(L_w/K_v , H(\mathcal{O}_{L_w})) \to H^1(L_w/K_v , H(L_w)).$$ By (ii), there exists $a \in H(\mathcal{O}_{L_w})$ such that $\xi(\sigma) = a^{-1} \sigma(a)$ for all $\sigma \in \Ga(L_w/K_v)$. Then $\sigma(ca^{-1}) = ca^{-1}$, implying that $$ca^{-1} \in G(\mathcal{O}_{L_w}) \cap G(K_v) = G(\mathcal{O}_v).$$ Furthermore,
$$
ah = (ca^{-1})^{-1} g \in H(L_w) \cap G(K_v) = H(K_v),
$$
and thus
$$
g \in G(\mathcal{O}_v) H(K_v).
$$
Since $g \in G(K_v)$ was arbitrary, (\ref{E:E1}) follows.

\end{proof}

\vskip2mm

\noindent {\it Proof of Theorem \ref{T-ToriT}.} In the case at hand, the connected component $H = G^{\circ}$ is a $K$-torus, which we will denote by $T$. According to Proposition \ref{P-ToriT1}, Condition (T) holds for $T$, so by Lemma \ref{L-1}, it is enough to check condition (\ref{E:E1}) for almost all $v \in V$. For this, we will use Lemma \ref{L-2}. Pick a finite Galois extension $L/K$ so that there exists a finite subset $C \subset G(L)$ such that $G(\overline{K}) = C T(\overline{K})$. Without loss of generality, we may assume that $T$ splits over $L$. Deleting from $V$ a finite number of places, we may also assume that for any $v \in V$ and any extension
$w \vert v$, the extension $L_w/K_v$ is unramified, we have the inclusion $C \subset G(\mathcal{O}_{L_w})$, and the subgroup $T(\mathcal{O}_{L_w})$ is a maximal bounded subgroup of $T(L_w)$. Then it only remains to verify condition (ii) of Lemma \ref{L-2}. Let $\pi \in K_v$ be a uniformizer. Since the extension $L_w/K_v$ is unramified, $\pi$ remains a uniformizer in $L_w$, yielding the following decomposition of $L_w^{\times}$ as $\mathrm{Gal}(L_w/K_v)$-module:
$$
L_w^{\times} = \langle \pi \rangle \times U_{L_w} \simeq \mathbb{Z} \times U_{L_w},
$$
where $U_{L_w} = \mathcal{O}_{L_w}^{\times}$ is the group of units in $L_w$. Let $X_*(T)$ be the group of cocharacters of $T$. Then we have the following isomorphisms of
$\mathrm{Gal}(L_w/K_v)$-modules:
$$
T(L_w) \simeq X_*(T) \otimes_{\mathbb{Z}} L_w^{\times} = X_*(T) \times (X_*(T) \otimes_{\mathbb{Z}} U_{L_w}).
$$
As $U_{L_w}$ is the maximal bounded subgroup of $L_w^{\times}$, we easily see that $X_*(T) \otimes_{\mathbb{Z}} U_{L_w}$ is the maximal bounded subgroup of $T(L_w)$, hence coincides with $T(\mathcal{O}_{L_w})$. Thus, the latter is a direct factor of $T(L_w)$ (as $\mathrm{Gal}(L_w/K_v)$-module), and the required injectivity of the map $$H^1(L_w/K_v , T(\mathcal{O}_{L_w})) \to H^1(L_w/K_v , T(L_w))$$ immediately follows, completing the proof. \hfill $\Box$

\vskip2mm

As an immediate consequence of Theorem \ref{T-ToriT} and the structure theory of reductive groups, we obtain

\begin{cor}\label{C-ToriT3}
Let $K$ be a finitely generated field and $V$ a divisorial set of places of $K$. Suppose $G$ is a connected reductive $K$-group and $T \subset G$ is a maximal $K$-torus. Then the normalizer $N_{G}(T)$ of $T$ in $G$ satisfies Condition {\rm (T)}.
\end{cor}

\vskip2mm

\noindent {\bf Remark 3.8.} It follows from \cite[Th\'eor\`eme 4.1]{Nisn1} that if $H$ is a semisimple algebraic $K$-group that has good reduction at $v$ and the extension $L_w/K_v$ is unramified, then the map $H^1(L_w/K_v , H(\mathcal{O}_{L_w})) \to H^1(L_w/K_v , H(L_w))$ has trivial kernel. This means that for a fixed semisimple $K$-group $H$ and a fixed finite Galois extension $L/K$, this map has trivial kernel for almost all $v \in V$, which is sufficient for analyzing Condition (T) by the method described above.

\section{Finiteness results for Tate-Shafarevich groups of tori}\label{S-FinSha}

In this section, we establish some finiteness results for the Tate-Shafarevich groups of algebraic tori over finitely generated fields (see Theorems \ref{T-FinSha1} and \ref{T-Sha2}). To fix notations, given a field $K$, a $K$-torus $T$, and a set $V$ of discrete valuations of $K$, we define
$$
\Sha^i(T,V) := \ker\left(H^i(K,T) \to \prod_{v \in V} H^i(K_v, T) \right),
$$
where $K_v$ denotes the completion of $K$ at $v \in V$. 

Our main result in this section is Theorem \ref{T-FinSha}, which we restate here for the reader's convenience.


\begin{thm}\label{T-FinSha1}
Let $K$ be a finitely generated field and $V$ be a divisorial set of places of $K$. Then for any algebraic $K$-torus $T$, the Tate-Shafarevich group
$$
\Sha^1(T,V) = \ker \left(H^1 (K,T) \to \prod_{v \in V} H^1(K_v, T) \right)
$$
is finite.
\end{thm}

We will give two proofs of this result. Our first proof, which covers the general case, makes use of adeles and relies on some of the techniques developed in \S\ref{S-CondT}. The second proof requires the additional assumption that the degree $n = [K_T:K]$, where $K_T$ is the minimal splitting field of $T$ inside a fixed separable closure $K^{\rm sep}$ of $K$, is prime to ${\rm char}~K$; however, the argument reveals important connections with finiteness results for \'etale and unramified cohomology, and is thus applicable in other situations.

\vskip2mm

\noindent {\it First proof of Theorem \ref{T-FinSha1}.} For a finite Galois extension $L/K$ and any $i \ge 1$ we define
$$
\text{{\brus SH}}^i(L/K, T, V) := \ker\left(H^i(L/K, T(L)) \to \prod_{v \in V} H^i(L_w/K_v , T(L_w))\right),
$$
where, in the product on the right, we choose, for each $v \in V$, a single extension $w \vert v$ in $V^L.$  The main ingredient in our first proof of Theorem \ref{T-FinSha1} is the following statement.
\begin{prop}\label{PP:1}
For any finite Galois extension $L/K$ and any $i \ge 1$, the group $\text{{\brus SH}}^i(L/K, T, V)$ is finite.
\end{prop}

First, let us show how this proposition implies Theorem \ref{T-FinSha1}. Let $L = K_T$ be the minimal splitting field of $T$. The inflation-restriction exact
sequence
$$
0 \to H^1(L/K , T(L)) \to H^1(K , T) \to H^1(L , T),
$$
combined with the fact that $H^1(L , T) = 0$ as $T$ is $L$-split (Hilbert's Theorem 90), enables us to canonically identify $H^1(K , T)$ with $H^1(L/K , T)$ (via the inverse of the inflation map). Similarly, we can canonically identify $H^1(K_v , T)$ with $H^1(L_w/K_v , T)$ for any extension $w \vert v.$ It follows that
$$
\text{{\brus SH}}^1(T , V) = \text{{\brus SH}}^1(L/K, T, V),
$$
which is finite by Proposition \ref{PP:1}.

\vskip2mm

We begin our proof of Proposition \ref{PP:1} with the following general lemma.
\begin{lemma}\label{LL:1}
In the above notations, the group $\text{{\brus SH}}^i(L/K, T, V)$ coincides with
$$
Q^i(L/K, T, V) := \ker\left(H^i(L/K , T(L)) \to H^i(L/K , T(\mathbb{A}(L, V^L))) \right),
$$
where $V^L$ consists of all extensions to $L$ of places in $V$.
\end{lemma}
\begin{proof}
First, recall that for any $v \in V$, the product ${\prod_{w \vert v} L_w}$ is identified as $\mathrm{Gal}(L/K)$-module with $L \otimes K_v$ (cf. \S\ref{S-CondT} and \cite[Ch. II, \S10]{ANT}). Therefore, by Shapiro's Lemma, we have a natural identification of $H^i(L/K , \prod_{w \vert v} T(L_w))$ with $H^i(L_w/K_v , T(L_w))$ for {\it any} extension $w \vert v$. Thus, the natural embedding of $T(\mathbb{A}(L, V^L))$ into $\prod_{w \in V^L} T(L_w)$ yields the inclusion
$$
Q^i(L/K, T, V) \subset \text{{\brus SH}}^i(L/K, T, V).
$$

Before establishing the opposite inclusion, let us first recall the well-known description of the cohomology of the adelic group $T(\mathbb{A}(L, V^L)).$ For each finite set $S \subset V$, let $\tilde{S} \subset V^L$ be the set of all extensions to $L$ of places in $S$, and define
$$
T(\mathbb{A}(L, V^L, \tilde{S})) = \prod_{v \in S}(\prod_{w \vert v} T(L_w)) \times \prod_{v \notin S} (\prod_{w \vert v} T(\mathcal{O}_{L_w}))
$$
Then $T(\mathbb{A}(L, V^L, \tilde{S}))$ is stable under the action of $\Ga(L/K)$ and $T(\mathbb{A}(L, V^L))$ is the direct limit of the $T(\mathbb{A}(L, V^L, \tilde{S}))$ as $S$ runs over the finite subsets of $V$ (in fact, without loss of generality, we may assume that the sets $S$ consist of places such that, for each $v \notin S$, the extension $L_w/K_v$ is unramified and $T$ has good reduction at $v$). Since the cohomology of finite groups commutes with direct limits, it follows that
$$
H^i (L/K, T(\mathbb{A}(L, V^L))) = \lim_{\overrightarrow{ \ S \ }} H^i (L/K, T(\mathbb{A}(L, V^L, \tilde{S}))).
$$
Furthermore, since for each $v$, the corresponding product is stable under $\Ga(L/K)$, we have
$$
H^i (L/K, T(\mathbb{A}(L, V^L, \tilde{S}))) = \prod_{v \in S} H^i(L/K, \prod_{w \vert v} T(L_w)) \times \prod_{v \notin S} H^i(L/K, \prod_{w \vert v} T(\mathcal{O}_{L_w})).
$$
By Shapiro's Lemma, the latter group is naturally identified with
$$
\prod_{v \in S} H^i(L_w/K_v, T(L_w)) \times \prod_{v \notin S} H^i(L_w/K_v, T(\mathcal{O}_{L_w})),
$$
where for each $v \in V$, we have chosen the same extension $w \vert v$ as in the definition of $\text{{\brus SH}}^i(L/K, T, V).$

From this discussion, we see that to prove the inclusion $\Sha^i(L/K, T,V) \subset Q^i(L/K, T, V)$, it is enough to show that the map
$$
\iota_v \colon H^i(L_w/K_v, T(\mathcal{O}_{L_w})) \longrightarrow H^i(L_w/K_v, T(L_w))
$$
is injective for almost all $v$. We will establish this using a slight modification of argument used in the proof of Theorem \ref{T-ToriT}.

Let $K_T$ be the minimal splitting field of $T$, and let $F = K_T \cdot L$ be the compositum of $K_T$ and $L$ inside a fixed separable closure $K^{\rm sep}$ of $K$. Then for all but finitely many $v \in V$, the extension $F/K$ is unramified
at $v$ and $T(\mathcal{O}_{F_u})$, where $u \vert v$, is a maximal bounded subgroup of $T(F_u)$. It turns out that for such $v$, the map $\iota_v$ is
injective. Indeed, as we have seen in the proof of Theorem \ref{T-ToriT}, we have an isomorphism of $\mathrm{Gal}(F_u/K_v)$-modules
$$
T(F_u) \simeq X_*(T) \times T(\mathcal{O}_{F_u}).
$$
Taking the fixed points under $\mathrm{Gal}(F_u/L_w)$, we obtain the decomposition
$$
T(L_w) \simeq M \times T(\mathcal{O}_{L_w}) \ \ \text{where} \ \ M = X_*(T)^{\mathrm{Gal}(F_u/L_w)}.
$$
Thus, $T(\mathcal{O}_{L_w})$ is a direct factor of $T(L_w)$ as $\mathrm{Gal}(L_w/K_v)$-module, and the injectivity of $\iota_v$ follows, completing the proof.
\end{proof}

\vskip2mm

Next, as in \S\ref{S-CondT}, for $w \in V^L$, let $\tilde{\Delta}_w$ denote the maximal bounded subgroup of $T(L_w)$, and set
$$
\tilde{\Delta} = \prod_{w \in V^L} \tilde{\Delta}_w.
$$
As we have seen previously, $\tilde{\Delta}$ embeds into $T(\mathbb{A}(L, V^L))$, yielding a map
$$
H^i(L/K , \tilde{\Delta}) \stackrel{\psi(i)}{\longrightarrow} H^i(L/K , T(\mathbb{A}(L , V^L))).
$$
We also have a map
$$
H^i(L/K , T(L)) \stackrel{\varphi(i)}{\longrightarrow} H^i(L/K , T(\mathbb{A}(L , V^L))),
$$
induced by the embedding $T(L) \hookrightarrow T(\mathbb{A}(L , V^L))$. We define the subgroup of {\it weakly unramified} elements in $H^i(L/K, T(L))$ to be
$$
W^i(L/K, T, V) = \varphi(i)^{-1}(\im \psi(i)).
$$
A bit more concretely, it is easy to see that $\alpha \in H^i(L/K , T)$ is {weakly unramified} if for any $v \in V$, its image
in $H^i(L_w/K_v , T)$ under
the restriction map lies in the image of the map
$$
H^i(L_w/K_v , \tilde{\Delta}_w) \to H^i(L_w/K_v , T(L_w))
$$
for some (equivalently, any) extension $w \vert v$.

Since we obviously have $Q^i(L/K, T, V) \subset W^i(L/K, T, V)$, it follows from Lemma \ref{LL:1} that for the proof
of Proposition \ref{PP:1}, it is enough to establish
\begin{prop}\label{PP:2}
With the preceding notations, the group $W^i(L/K, T, V)$ is finite.
\end{prop}
\begin{proof}
We begin with the following two exact sequences:
\begin{equation}\label{EE:2}
1 \to E \longrightarrow T(L) \times \tilde{\Delta}  \stackrel{\pi}{\longrightarrow} H \to 1,
\end{equation}
where $E = T(L) \cap \tilde{\Delta}$ is embedded in $T(L) \times \tilde{\Delta}$ via $e \mapsto (e , e^{-1})$ and $H = T(L) \cdot \tilde{\Delta} \subset T(\mathbb{A}(L , V^L))$ with $\pi$ being the product map, and
\begin{equation}\label{EE:3}
1 \to H \longrightarrow T(\mathbb{A}(L , V^L)) \longrightarrow T(\mathbb{A}(L , V^L))/H \to 1.
\end{equation}
Then (\ref{EE:2}) yields the following exact sequence in cohomology
\begin{equation}\label{EE:5}
H^i(L/K , E) \longrightarrow H^i(L/K , T(L)) \times H^i(L/K , \tilde{\Delta}) \stackrel{\overline{\varphi(i)} + \overline{\psi(i)}}{\longrightarrow} H^i(L/K , H),
\end{equation}
where $\overline{\varphi(i)}$ and $\overline{\psi(i)}$ are the same maps as $\varphi(i)$ and $\psi(i)$ but with the target being $H^i(L/K , H)$ instead of
$H^i(L/K , T(\mathbb{A}(L , V^L)))$, and (\ref{EE:3}) yields the exact sequence
$$
H^{i-1}(L/K , T(\mathbb{A}(L , V^L))/H) \stackrel{\varepsilon}{\longrightarrow}  H^i(L/K , H) \stackrel{\omega}{\longrightarrow} H^i(L/K ,
T(\mathbb{A}(L , V^L))).
$$
We note that
\begin{equation}\label{EE:4}
\varphi(i) + \psi(i) = \omega \circ (\overline{\varphi(i)} + \overline{\psi(i)}).
\end{equation}
On the other hand, if $p \colon H^i(L/K , T(L)) \times H^i(L/K , \tilde{\Delta}) \to H^i(L/K , T(L))$ denotes the canonical projection, then clearly
$$
W^i(L/K, T, V) = p(\ker (\varphi(i) + \psi(i))),
$$
so it is enough to prove the finiteness of $\ker (\varphi(i) + \psi(i))$. As we observed at the start of \S\ref{S-CondT}, the set $V^L$ is a divisorial set of places of $L$, so Proposition \ref{P-ToriT2} implies that the quotient $T(\mathbb{A}(L , V^L))/H$ is a finitely generated abelian group.
Consequently, by Lemma \ref{L-FiniteCoh}, the group $H^{i-1}(L/K , T(\mathbb{A}(L , V^L))/H)$ is finite for $i \ge 2$,. For $i = 1$, it may be infinite, but it is still finitely generated. Since $H^i(L/K , H)$ has finite exponent (cf. \cite[Ch. II, Corollary 1.31]{MilneCFT}), we see that $\im \varepsilon$ is finite for all $i \ge 1$, and hence $\ker \omega$ is finite. So, it follows from (\ref{EE:4}) that it is enough to prove the finiteness of $\ker(\overline{\varphi(i)} + \overline{\psi(i)})$. However, we observed in the proof of Proposition \ref{P-ToriT2} that $E$ is a finitely generated abelian group and hence $H^i(L/K, E)$ is finite by Lemma \ref{L-FiniteCoh}. Thus, the required fact follows immediately from the exact sequence (\ref{EE:5}).
\end{proof}

Let us point out that one can somewhat streamline the proof of Proposition \ref{PP:2} using the results on Condition (T) from \S\ref{S-CondT}. More precisely, applying Proposition \ref{P-ToriT1} to $T$ over $L$ with the set of places $V^L$,  
we can find a co-finite subset $V' \subset V$ such that $T(\mathbb{A}(L , (V')^L)) = T(L) \tilde{\Delta}'$, where
$$
\tilde{\Delta}' = \prod_{w \in (V')^L} \tilde{\Delta}_w.
$$
Clearly, $W^i(L/K, T, V')$ contains $W^i(L/K, T, V)$, so it is enough to prove the finiteness of the former. Thus, replacing $V$ with $V'$, we may assume that $H = T(L) \tilde{\Delta}$ coincides with $T(\mathbb{A}(L , V^L))$. In this case, in place of (\ref{EE:2}), we have the short exact sequence
$$
1 \to E \longrightarrow T(L) \times \tilde{\Delta} \stackrel{\pi}{\longrightarrow} T(\mathbb{A}(L , V^L)) \to 1,
$$
which leads to the exact sequence
$$
H^i(L/K , E) \longrightarrow H^i(L/K , T(L)) \times H^i(L/K , \tilde{\Delta}) \stackrel{\varphi(i) + \psi(i)}{\longrightarrow} H^i(L/K , T(\mathbb{A}(L , V^L))). $$
As above, the group $H^i(L/K , E)$ is finite, so the required finiteness of $\ker(\varphi(i) + \psi(i))$ immediately follows.

\vskip2mm

\noindent {\bf Remark 4.5.} The above argument shows that the finiteness of the Tate-Shafarevich group of any torus over a finitely generated field with respect to a divisorial set of places is a direct consequence of the finite generation of the relevant unit and class groups (over a suitable extension of the base field).

\addtocounter{thm}{1}

\vskip3mm

\noindent {\it Second proof of Theorem \ref{T-FinSha1}.} As remarked above, in our second proof, we will make the additional assumption that  $n = [K_T:K]$ is prime to ${\rm char}~K$. We will proceed by first reducing Theorem \ref{T-FinSha1} to a statement about the finiteness of a certain group of unramified cohomology, and then proving this statement in Theorem \ref{P-H2Finite} below.


For the argument, we fix a model $X$ of $K$ over $\Z$ (if ${\rm char}~K = 0$) or over a finite field (if ${\rm char}~K > 0$) such that $V$ is associated with the prime divisors of $X$. After possibly shrinking $X$, we can assume that $X$ is smooth (over an open subset of ${\rm Spec}(\Z)$ in the first case and over a finite field in the second), $n$ is invertible on $X$, and $T$ extends to torus $\mathbb{T}$ over $X$ (so that, in particular, the generic fiber of $\mathbb{T}$ is $T$).

First, the inflation-restriction sequence
$$
0 \to H^1(K_T/K , T) \to H^1(K , T) \to H^1(K_T , T),
$$
combined with the fact that
$H^1(K_T , T) = 0$ as $T$ is $K_T$-split (Hilbert's Theorem 90) enables us to canonically identify $H^1(K , T)$ with $H^1(K_T/K , T)$. 
This, in particular, implies that $n H^1(K , T) = 0$. Furthermore, since $n$ is prime to the characteristic of $K$, we can consider the Kummer sequence
$$
1 \to M \to T \stackrel{\times n}{\longrightarrow} T \to 1,
$$
where $M = T[n]$ is the $n$-torsion in $T$. Then, since $H^1(K , T)$ is annihilated by $n$, the corresponding exact sequence in Galois cohomology
$$
H^1(K , T) \stackrel{\times n}{\to} H^1(K , T) \to H^2(K , M)
$$
yields an embedding
$$
\psi_K \colon H^1(K , T) \to H^2(K , M).
$$
Note that there are similar embeddings $\psi_L$ for all extensions $L/K$, and these behave functorially. We thus obtain the commutative diagram
$$
\xymatrix{H^1(K,T) \ar[r]^{\psi_K} \ar[d] & H^2(K,M) \ar[d] \\ \displaystyle{\prod_{v \in V} H^1 (K_v,T)} \ar[r]^{\Psi} & \displaystyle{\prod_{v \in V} H^2(K_v,M)}}
$$
where $\Psi = \displaystyle{\prod_{v \in V} \psi_{K_v}}.$ It follows that $\psi_K$ gives an embedding of $\Sha^1(T,V)$ into $$\Sha^2(M , V) := \ker\left( H^2(K , M) \to \prod_{v \in V} H^2(K_v , M) \right),$$ and it is enough to prove the finiteness of the latter. For this, we relate $\Sha^2(M,V)$ to unramified cohomology.


More precisely, it follows from our construction that for any  $v \in V$, the field extension $K_T/K$ is unramified at $v$. Since, in addition, $n$ is invertible on $X$, the $\Ga(K_v^{sep}/K_v)$-module $M = T[n]$ is unramified (and can be naturally identified with the $\Ga((K^{(v)})^{sep}/K^{(v)})$-module $\underline{T}^{(v)}[n]$, where $\underline{T}^{(v)}$ denotes the reduction of $T$ at $v$). So, there exists a residue map
$$
\tilde{\partial}_v \colon H^2(K_v , M) \to H^1(K^{(v)} , M(-1)),
$$
and the residue map $\partial_v \colon H^2(K , M) \to H^1(K^{(v)} , M(-1))$ mentioned in \S\ref{S-Introduction}
is then the composition of $\tilde{\partial}_v$ with the restriction map $H^2(K , M) \to H^2(K_v , M)$ (see  \cite[Ch. II, \S7]{GMS}). It follows that for any $x \in \Sha^2(M , V)$, all residues $\partial_v(x)$, $v \in V$ are trivial. In other words,
$\Sha^2(M , V)$ is contained in the unramified cohomology group
$$
H^2(K , M)_V := \bigcap_{v \in V} \ker \partial_v.
$$
So, to complete the proof of Theorem \ref{T-FinSha1}, it remains to establish
\begin{thm}\label{P-H2Finite}
With the preceding notations, $H^2(K, M)_V$ is finite.
\end{thm}

For the proof of Theorem \ref{P-H2Finite}, it will be convenient to introduce the following notations. We let $\mathbb{M}$ denote the $n$-torsion subscheme $\mathbb{T}[n]$ of $\mathbb{T}$ over $X$. As usual, we will use the same notation for the associated \'etale sheaf. Next, for $P \in X$, we denote
by $\mathcal{O}_{X,P}$ the local ring of $X$ at $P$. We then define
$$
D(X) = {\rm Im}(\he^2(X, \mathbb{M}) \to H^2(K, M))
$$
and
$$
D(X,P) = {\rm Im}(\he^2({\rm Spec}(\mathcal{O}_{X,P}), \mathbb{M}) \to H^2(K,M)),
$$
where the maps are the natural ones induced by passage to the generic point. 
An important ingredient in the proof of Theorem \ref{P-H2Finite} is the following lemma.

\begin{lemma}\label{L-Purity}
We have equalities
$$
D(X) = \bigcap_{P \in X^{(1)}} D(X,P) = H^2(K, M)_V,
$$
where $X^{(1)}$ denotes the set of codimension 1 points (prime divisors) of $X$.
\end{lemma}
\begin{proof}
To prove the second equality, it suffices to show that for every point $P \in X^{(1)}$ and the corresponding discrete valuation $v$ of $K$, the kernel $\ker \partial_v$ of the residue map $\partial_v$ considered above coincides with $D(X , P)$. Indeed, it is known (see \cite[\S3.3]{CT-SB} and references therein) that the residue map
$$
\delta_P \colon H^2(K , M) \to H^1(K^{(v)} , M(-1))
$$
arising from absolute purity for discrete valuation rings applied to the locally constant constructible \'etale sheaf of $\Z/ n \Z$-modules
on $\mathrm{Spec}(\mathcal{O}_{X,P})$ associated with $\mathbb{M}$, coincides up to sign with $\partial_v$. On the other hand, $\delta_P$ fits into the following exact sequence that is derived from the localization sequence in \'etale cohomology
$$
\cdots \to \he^2({\rm Spec}(\mathcal{O}_{X,P}), \mathbb{M}) \to H^2(K, M) \stackrel{\delta_P} \to H^1(K^{(v)}, M(-1)) \to \cdots,
$$
from which the required fact follows.

Let us now turn to the first equality. One of the crucial ingredients needed for the argument is the truth of the absolute purity conjecture for regular noetherian schemes, which was established by Gabber (see \cite{Fuj}). As observed in \cite[\S3.4]{CT-SB}, it implies that for any open subscheme $U \subset X$ such that ${\rm codim}_X (X \setminus U) \geq 2$, the restriction map $\he^2(X, \mathbb{M}) \to \he^2(U, \mathbb{M})$ is surjective. Thus, to prove the first equality, it is enough to show that for every
$$\gamma \in \bigcap_{P \in X^{(1)}} D(X , P),$$
there exists an open subscheme $U_{\gamma} \subset X$  (depending on $\gamma$) such that $\mathrm{codim}_X (X \setminus U_{\gamma}) \geq 2$ and $\gamma \in D(U_{\gamma})$. This is done by the following (relatively) standard argument (see, for example, \cite[Theorem 3.8.2]{CT-SB}, \cite[Proposition 6.8]{CTS}, and \cite[Corollary A.8]{GP}),  which actually yields an open subscheme $U_{\gamma} \subset X$ containing $X^{(1)}$ (implying that the codimension of its  complement is $\geq 2$) with $\gamma \in D(U_{\gamma})$.

First, according to [1, VII, 5.9], we have
\begin{equation}\label{E:Func1}
H^2(K , M) = \lim_{\overrightarrow{\ \ U \ \ }} \he^2(U , \mathbb{M}),
\end{equation}
where the limit is taken over all nonempty open affine subschemes $U$ of $X$. So, there exists such $U$ with $\gamma \in D(U)$, i.e. $\gamma$ is the image in $H^2(K , M)$ of some $\gamma^U \in \he^2(U , \mathbb{M})$. If $U$ contains $X^{(1)}$, we are done. Otherwise, the complement $X^{(1)} \setminus (X^{(1)} \cap U)$ consists of finitely many points. Then, in order to extend $U$ to a required open set $U_{\gamma}$ by an obvious inductive argument, it suffices to prove the following: for any $P \in X^{(1)} \setminus (X^{(1)} \cap U)$, there exists an open $\tilde{U} \subset X$ containing $U \cup \{ P \}$ such that $\gamma \in D(\tilde{U})$. By our assumption, $\gamma \in D(X , P)$, so there exists
an open affine neighborhood $W$ of $P$ such that $\gamma$ is the image in $H^2(K , M)$ of some $\gamma^P \in \he^2(W , \mathbb{M})$.  It follows from (\ref{E:Func1}) that there exists an open affine subset $W_0 \subset U \cap W$ such the images of $\gamma^U$ and $\gamma^P$ in $\he^2(W_0 , \mathbb{M})$ coincide. Let us show that there exists an open neighborhood $W' \subset W$ of $P$ such that $U \cap W' \subset W_0$. Indeed, since $P \notin  U$, we have $P \in X \setminus W_0$. As $P \in X^{(1)}$, the closure $\overline{\{ P \}}$ is an irreducible component of $X \setminus W_0$. Let $Z$ be the union of all other irreducible components. Then $W' := W
\cap (X \setminus Z)$ is an open neighborhood of $P$. Furthermore, the complement $X \setminus (U \cap W')$ is a closed subset that contains $P$ and $Z$, hence contains $X \setminus W_0$. Thus, $U \cap W' \subset W_0$, as required. Set $\tilde{U} = U \cup W'$, and let $\gamma'$ be the restriction of $\gamma^P$ to $W'$. Then $(\gamma^U , \gamma') \in  \he^2(U , \mathbb{M}) \oplus \he^2(W' , \mathbb{M})$ and the restrictions of $\gamma^U$ and $\gamma'$ to $U \cap W'$ coincide. Then by the Mayer-Vietoris sequence in \'etale cohomology (see, for example, \cite[Ch. 57, Lemma 57.50.1]{Stacks}), there exists $\tilde{\gamma} \in H^2_{et}(\tilde{U} , \mathbb{M})$ that restricts to $\gamma^U$ and $\gamma'$ on $U$ and $W'$, respectively. Clearly, $\tilde{\gamma}$ maps to $\gamma$, showing that $\gamma \in D(\tilde{U})$, as required.

\end{proof}

\vskip2mm

\noindent {\it Proof of Theorem \ref{P-H2Finite}.} In view of Lemma \ref{L-Purity}, it suffices to show that $\he^2(X, \mathbb{M})$ is finite.
This is established by the same argument as in the proof of \cite[Theorem 10.2]{CRR3}, which relies on Deligne's theorem for the higher direct images of constructible sheaves (see \cite[Th\'eor\`eme 1.1 in ``Th\'eor\`emes de finitude en cohomologie $\ell$-adique]{Del}), the Leray spectral sequence, and finiteness results for the \'etale cohomology of constructible sheaves over the spectrum of a finite field or the ring of $S$-integers in a number field.
$\Box$

\vskip2mm

To conclude this section, we would like to observe that the finiteness of $\Sha^1(T,V)$ yields the following statement concerning a subgroup with bounded torsion in $\Sha^2(T,V).$

\begin{thm}\label{T-Sha2}
Let $K$ and $V$ be as in Theorem \ref{T-FinSha1}. Then for any $K$-torus $T$ and any integer $\ell >0$ prime to $\mathrm{char}~K$,
the $\ell$-torsion subgroup ${}_{\ell}\Sha^2(T , V)$ is finite.
\end{thm}
\begin{proof}
Let $L = K_T$ be the minimal splitting field of $T$ inside a fixed separable closure $K^{\rm sep}$ of $K$. 
It is well-known (see, for example, \cite[Proposition 2.1]{Pl-R}) that $T$ can be embedded into an exact sequence of $K$-tori
\begin{equation}\label{E-Tori}
1 \to T \to T_0 \to T_1 \to 1,
\end{equation}
where $T_0$ is a quasi-split $K$-torus of the form $\mathbf{R}_{L/K}(\mathbb{G}_m)^s$ for some $s > 0$. Let us first show that ${}_{\ell}\Sha^2(T_0 , V)$ is finite. It is obviously enough to establish the finiteness of ${}_{\ell}\Sha^2(\mathcal{T} , V)$ for $\mathcal{T} = \mathbf{R}_{L/K}(\mathbb{G}_m)$. Now, by Shapiro's Lemma, we have
$$
H^2(K , \mathcal{T}) = H^2(L , \mathbb{G}_m) = \mathrm{Br}(L),
$$
so that $\text{{\brus SH}}(\mathcal{T} , V)$ can be identified with the kernel $\Omega$ of the natural map of Brauer groups
$$
\mathrm{Br}(L) \longrightarrow \prod_{w \in V^{L}} \mathrm{Br}(L_w),
$$
where $V^{L}$ consists of all extensions of places in $V$ to $L$. Clearly, ${}_{\ell}\Omega$ is contained in the $\ell$-torsion of unramified Brauer
group ${}_{\ell}\mathrm{Br}(L)_{V^{L}}$. Since $V^{L}$ is a divisorial set of places of the finitely generated field $L$ (see the discussion at the start of \S\ref{S-CondT}) and $\ell$ is prime to ${\rm char}~L$, the latter is finite (see \cite[Theorem 1]{CRR3} --- note that this is also a formal consequence of Theorem \ref{P-H2Finite} above). So, the required finiteness of ${}_{\ell}\Sha(T_0 , V)$ follows.

Next, since $H^1(F , T_0) = 0$ for any field extension $F/K$ by Hilbert's Theorem 90, the exact sequence (\ref{E-Tori}) gives rise to the following commutative diagram with exact rows:
$$
\xymatrix{0 \ar[r] & H^1(K,T_1) \ar[d] \ar[r]^{\alpha} & H^2(K,T) \ar[d] \ar[r]^{\beta} & H^2(K,T_0) \ar[d] \\ 0 \ar[r] & \displaystyle{\prod_{v \in V} H^1(K_v, T_1)} \ar[r] & \displaystyle{\prod_{v \in V} H^2(K_v, T)} \ar[r] & \displaystyle{\prod_{v \in V} H^2(K_v, T_0)}}
$$
Clearly, $\beta({}_{\ell}\Sha^2(T,V)) \subset {}_{\ell}\Sha^2(T_0, V)$, hence finite as we just showed. On the other hand,
$$
{}_{\ell}\Sha^2(T,V) \cap \ker \beta = {}_{\ell}\Sha^2(T,V) \cap {\rm Im}~\alpha = \alpha({}_{\ell}\Sha^1(T_1,V)).
$$
As we showed in Theorem \ref{T-FinSha1} above, $\Sha^1(T_1, V)$ is finite, so the finiteness of ${}_{\ell}\Sha^2(T,V)$ follows.
\end{proof}

\section{Finiteness results for unramified cohomology and applications}\label{S-UnramCoh}

In this section, we prove several finiteness results for the unramified cohomology of function fields of rational varieties and Severi-Brauer varieties over number fields. We then apply these statements to the framework developed in \cite{CRR-Spinor}
to establish some new cases of Conjectures 1 and 2.

To streamline the statements of our results in this section, we introduce the following condition. Suppose $K$ is a field and let $n > 1$ be an integer. We will say that a set $V$ of discrete valuations of $K$ satisfies condition (B) with respect to $n$ if

\vskip2mm

\noindent (B) \ $n$ is invertible in the residue fields $K^{(v)}$ for all $v \in V$.

\vskip2mm

\noindent Notice that if $K$ is a finitely generated field of characteristic 0 and $V$ is a divisorial set of places, then for any $n > 1$,
one can ensure that (B) holds by deleting finitely many places from $V$. On the other hand, if ${\rm char}~K = p > 0$, then (B) holds automatically for any $n$ prime to $p$ and any set of places $V$ of $K$. In any case, whenever $V$ satisfies (B) with respect to $n$, the unramified cohomology groups $H^i(K, \mu_n^{\otimes j})_V$ are defined for all $i \geq 1.$



\vskip2mm

We begin by considering the unramified cohomology of function fields of rational varieties over number fields.

\begin{thm}\label{T-PurelyTran}
Let $k$ be a number field and $K = k(x_1, \dots, x_r)$ be a purely transcendental extension of $k$ of transcendence degree $r \geq 1$.
Fix an integer $n > 1$. Suppose that $V$ is a divisorial set of places of $K$ satisfying {\rm (B)} with respect to $n$ and assume that $k$ contains a primitive $n$-th root of unity. Then

\vskip1mm

\noindent {\rm (a)} The unramified cohomology groups $H^i(K, \mu_n)_V$ are finite for $i \leq 3$.

\vskip1mm

\noindent {\rm (b)} \parbox[t]{16cm}{If $r \leq 2$, then the unramified cohomology groups $H^i(K, \mu_n)_V$ are finite for all $i \geq 1$.}
\end{thm}

\begin{proof} (a) The finiteness of $H^1(K, \mu_n)_V$ is well-known (see, for example, \cite[\S4.1]{CRR-Spinor} for a sketch of the argument) and the finiteness of $H^2(K, \mu_n)_V = {}_n\Br(K)_V$ was established in \cite[\S10]{CRR3} (note that this also follows from Theorem \ref{P-H2Finite} above). It remains to show the finiteness of $H^3(K, \mu_n)_V.$

First, viewing $K$ as the function field of $\mathbb{P}_k^r$, let $V_0$ be the set of discrete valuations of $K$ associated with all of the prime divisors of $\mathbb{P}_k^r.$ It is well-known that the natural map $H^3(k, \mu_n) \to H^3(K, \mu_n)$ induces an isomorphism $H^3(k, \mu_n) \stackrel{\sim}{\longrightarrow} H^3(K, \mu_n)_{V_0}$ (see, for example, \cite[Theorem 4.1.5]{CT-SB}). Since $H^3(k, \mu_n)$ is finite by the results of Poitou-Tate, we thus obtain the finiteness of $H^3(K, \mu_n)_{V_0}.$

Next, suppose that $V$ is a divisorial set of places of $K$. Denote by $\mathcal{O}_k$ the ring of integers of $k$. Since, as observed in \S\ref{S-Introduction}, any two divisorial sets are commensurable, it follows that after possibly deleting a finite number of places from $V$, we may assume that $V = V(U)$ is the set of discrete valuations of $K$ associated with the prime divisors of a smooth open subscheme $U \subset \mathbb{P}^r_S$, where $S$ is an open subscheme of ${\rm Spec}(\mathcal{O}_k)$ with $n$ invertible on $S$. Moreover, possibly shrinking $U$ (which reduces $V$ by a finite number of places), we can
assume that $\mathbb{P}^r_S \setminus U$ is pure of codimension one. Let
$$
\mathbb{P}^r_S \setminus U = \bigcup_{j \in J} Y_j
$$
be the decomposition into irreducible components and denote by $\kappa_j$ the function field of $Y_j.$ By \cite[Proposition 6.4]{CRR-Spinor} (see also \cite[\S2, p. 36]{CTFinitude}), we have an exact sequence of unramified cohomology groups
\begin{equation}\label{E-UnramSequence}
0 \to H^3(K, \mu_n)_{V(\mathbb{P}^r_S)} \to H^3(K, \mu_n)_{V(U)} \to \bigoplus_{j \in J} H^2(\kappa_j, \mu_n)_{V(Y_j)},
\end{equation}
where for a scheme $X$, we set $V(X)$ to be the set of discrete valuations of the function field $k(X)$ associated with all of the prime divisors of $X$. The inclusion
$$
H^3(K, \mu_n)_{V(\mathbb{P}^r_S)} \subset H^3(K, \mu_n)_{V_0}
$$
and the preceding remarks imply that $H^3(K, \mu_n)_{V(\mathbb{P}^r_S)}$ is finite. Furthermore, as shown in \cite[\S10]{CRR3}, each of terms $H^2(\kappa_j, \mu_n)_{V(Y_j)}$ is finite.
Thus, $H^3(K, \mu_n)_{V(U)}$ is finite, which gives the finiteness of $H^3(K, \mu_n)_V$, as needed.

\vskip2mm

\noindent (b) We only need to show the finiteness of $H^i(K, \mu_n)_V$ for $i \geq 4.$ Proceeding as in part (a), we have an exact sequence
$$
0 \to H^i(K, \mu_n)_{V(\mathbb{P}^r_S)} \to H^i(K, \mu_n)_{V(U)} \to \bigoplus_{j \in J} H^{i-1}(\kappa_j, \mu_n)_{V(Y_j)}
$$
for all $i \geq 4.$ Again, the results of Poitou-Tate imply that the groups $H^i(K, \mu_n)_{V(\mathbb{P}^r_S)}$ are finite. If $r = 1$, then the fields $\kappa_j$ in the terms on the right are number fields, so the groups $H^{t}(\kappa_j, \mu_n)_{V(Y_j)}$ are finite for all $t \geq 3$ by Poitou-Tate. If $r = 2$, then the $\kappa_j$ have transcendence degree 1 over $k$, in which case the finiteness of $H^{t}(\kappa_j, \mu_n)_{V(Y_j)}$ for all $t \geq 3$ follows
from \cite[\S4 and Theorem 6.3]{CRR-Spinor}.
Consequently, the groups $H^i(K, \mu_n)_{V(U)}$ are finite, which yields the finiteness of $H^i(K, \mu_n)_V$ for all $i \geq 4$, as claimed.

\end{proof}

\noindent {\bf Remark 5.2.} Although it is possible to formulate a similar result for purely transcendental extensions of global function fields, in the study of unramified cohomology, it is more traditional to consider function fields of varieties over {\it finite} fields. So, let $k = \mathbb{F}_q$ be a finite field of characteristic $p > 0$ and suppose $n > 1$ is an integer relatively prime to $p$ such that $k$ contains a primitive $n$-th root of unity. Then part (a) of the theorem does not require any changes, and part (b) holds for $r \leq 3.$ The main case to consider is the unramified cohomology in degree 4 of the field $K = \mathbb{F}_q(x_1, x_2, x_3).$ Proceeding as above, the argument ultimately boils down to the finiteness of the unramified cohomology in degree 3 of the function field of a smooth surface over a finite field, a proof of which is given in \cite[\S7]{CRR-Spinor}.

\vskip2mm

\addtocounter{thm}{1}





Building on Theorem \ref{T-PurelyTran}, one would like to understand the unramified cohomology of the function fields of geometrically rational varieties without rational points. The next proposition treats some Severi-Brauer varieties over global fields:
while the first part is straightforward, the second requires input from Kahn's analysis of the motivic cohomology of Severi-Brauer varieties.

\begin{prop}\label{P-SB}
Let $k$ be a global field and $X$ be the Severi-Brauer variety over $k$ associated with a central division algebra $D$ over $k$ of degree $\ell$. Denote by $K = k(X)$ the function field of $X$ and let $V(X)$ be the set of geometric places of $K$, i.e. the discrete valuations of $K$ corresponding to the prime divisors of $X$.

\vskip2mm

\noindent {\rm (a)} \parbox[t]{16.5cm}{For any integer $n>1$ that is prime to $\ell$ and the characteristic of $k$, the unramified cohomology groups $H^i(K, \mu_{n}^{\otimes j})_{V(X)}$ are \emph{finite} for all $i \geq 3$ and all $j$ if $k$ is a number field, and \emph{trivial} if $k$ has positive characteristic.}

\vskip2mm

\noindent {\rm (b)} \parbox[t]{16.5cm}{If $\ell$ is a prime $\neq {\rm char}~k$, then group $H^3(K, \mu_{\ell^t}^{\otimes 2})_{V(X)}$ is \emph{finite} for all $t \geq 1$ if $k$ is a number field, and \emph{trivial} if $k$ has positive characteristic.}

\end{prop}
\begin{proof}
(a) Let $k'$ be a maximal separable subfield of $D$. Then $X_{k'} \simeq \mathbb{P}_{k'}^{\ell - 1}$. Set $K' = k'(X_{k'})$. Since $n$ is prime to $\ell$, the restriction map $H^i(K , \mu_n^{\otimes j}) \to H^i(K' , \mu_n^{\otimes j})$ is injective, giving rise to an inclusion
$$
H^i(K , \mu_n^{\otimes j})_{V(X)} \hookrightarrow H^i(K' , \mu_n^{\otimes j})_{V(X_{k'})}.
$$
It is well-known that $H^i(K', \mu_{n}^{\otimes j})_{V(X_{k'})} \simeq H^i(k', \mu_{n}^{\otimes j})$ (see, for example, \cite[Theorem 4.1.5]{CT-SB}). The latter group is finite for $i \geq 3$ if $k$ is a number field by results of Poitou-Tate and trivial by considerations of cohomological dimension if $k$ has positive characteristic. This yields our claim.

For (b), let us first show that the natural map
$$
\eta \colon H^3(k, \Q_{\ell}/\Z_{\ell}(2)) \to H^3(K, \Q_{\ell}/\Z_{\ell}(2))_{V(X)}
$$
is surjective (cf. \cite{Pirutka}). We recall that Kahn has constructed (over any field) a complex
$$
0 \to {\rm coker}(\eta) \to {\rm coker}\left(CH^2(X) \to CH^2(\overline{X})^{\Ga(k^{sep}/k)}\right)[\ell^{\infty}] \stackrel{d}{\longrightarrow} \Br(k)[\ell^{\infty}]
$$
that is exact except perhaps at the middle term (see \cite[Corollary 7.1]{Kahn}), where $(\cdot)[\ell^{\infty}]$ denotes the $\ell$-primary part of the corresponding abelian group and $CH^2$ the Chow group of codimension 2 cycles modulo rational equivalence. If $\ell = 2$, then $CH^2(\overline{X}) = 0$, and we immediately obtain ${\rm coker}(\eta) = 0$. Thus, we may assume that $\ell \geq 3.$ In that case, $${\rm coker}\left(CH^2(X) \to CH^2(\overline{X})^{\Ga(k^{sep}/k)}\right) \simeq \Z/ \ell \Z$$ by \cite[Corollary 8.7.2]{MS}. Furthermore, according to \cite[Theorem 7.1]{KahnCellular}, we have $d(1) = 2[D] \neq 0$, where $1 \in \Z/ \ell \Z$ is identified with a generator of ${\rm coker}\left(CH^2(X) \to CH^2(\overline{X})^{\Ga(k^{sep}/k)}\right)$ and $[D]$ denotes the class of $D$ in $\Br(k).$ It follows that ${\rm coker}(\eta) = 0$. Thus, in all cases $\eta$ is surjective. Now,
$H^3(k, \Q_{\ell}/\Z_{\ell}(2))$ is finite if $k$ is a number field by results of Poitou-Tate and trivial by considerations of cohomological dimension if $k$ has positive characteristic, hence the same is true for $H^3(K, \Q_{\ell}/\Z_{\ell}(2))_{V(X)}$. To conclude the argument, recall that a well-known consequence of the Merkurjev-Suslin theorem is that the natural map $$H^3(K, \mu_{\ell^t}^{\otimes 2}) \to H^3(K, \Q_{\ell}/\Z_{\ell}(2))$$
is injective for all $t \geq 1$ (see \cite[18.4]{MS}). This gives an inclusion
$$
H^3(K, \mu_{\ell^t}^{\otimes 2})_{V(X)} \hookrightarrow H^3(K, \Q_{\ell}/\Z_{\ell}(2))_{V(X)},
$$
completing the proof.

\end{proof}

We note that Proposition \ref{P-SB} and considerations analogous to the ones appearing in the proof of Theorem \ref{T-PurelyTran} yield the following statement.

\begin{cor}\label{C-SB}
Let $k$ be a global field and $m > 1$ be an integer relatively prime to ${\rm char}~k$ such that $k$ contains a primitive $m$-th root of unity. Furthermore, let $K = k(X)$ be the function field of a Severi-Brauer variety $X$ associated with a central division algebra $D$ over $k$ of degree $\ell$, and let $V$ be a divisorial set of places of $K$ satisfying condition ${\rm (B)}$ with respect to $m.$ If either $m$ is relatively prime to $\ell$ or $\ell$ is a prime number $\neq {\rm char}~k$, then the unramified cohomology groups $H^i(K, \mu_m^{\otimes 2})_V$ are finite for $i \leq 3$.
\end{cor}

In \cite{CRR-Spinor}, we discovered connections between finiteness properties of unramified cohomology with $\mu_2$-coefficients and Conjectures 1 and 2 for certain groups. We will now apply the preceding results to this framework in order to establish several new cases of the conjectures.

We begin with the following statement for spinor, special orthogonal, and special unitary groups, which relies on the finiteness of unramified cohomology in {\it all} degrees with $\mu_2$-coefficients.

\begin{thm}\label{T-Thm1}
Let $k$ be a number field, $K = k(x_1,x_2)$ a purely transcendental extension of $k$ of transcendence degree $2$, and $V$ any divisorial set of places of $K$.

\vskip2mm

\noindent {\rm (a)} \parbox[t]{16.5cm}{For any $n \geq 5,$ the set of $K$-isomorphism classes of spinor groups $G = {\rm Spin}_n(q)$ of nondegenerate quadratic forms in $n$ variables over $K$ that have good reduction at all $v \in V$ is finite.}

\vskip2mm

\noindent {\rm (b)} \parbox[t]{16.5cm}{For any $n \geq 5$ and $G = {\rm SO}_n(q)$, with $q$ a nondegenerate quadratic form in $n$ variables over $K$, the global-to-local map
$$
\lambda_{G,V} \colon H^1 (K,G) \to \prod_{v \in V} H^1(K_v,G)
$$
is proper. In particular, $\Sha(G,V)$ is finite.}

\vskip2mm

\noindent {\rm (c)} \parbox[t]{16.5cm}{Fix a quadratic extension $L/K$ and let $n \geq 2.$ Then the number of $K$-isomorphism classes of special unitary groups $G= {\rm SU}_n(L/K,h)$ of nondegenerate hermitian $L/K$-forms in $n$ variables that have good reduction at all $v \in V$ is finite. Moreover, the global-to-local map
$$
\lambda_{G, V} \colon H^1(K,G) \to \prod_{v \in V} H^1(K_v, G)
$$
is proper. In particular, $\Sha(G,V)$ is finite.}

\end{thm}
\begin{proof}
After deleting finitely many places from $V$, we may assume that $V$ satisfies condition (B) with respect to 2, so that
the unramified cohomology groups $H^i(K,\mu_2)_V$ are finite for all $i \geq 1$ by Theorem \ref{T-PurelyTran}. The statements are then obtained by imitating the proofs of Theorems 1.1, 1.3, 8.1, and 8.4 in \cite{CRR-Spinor}, respectively.
\end{proof}

\noindent (Of course, part (b) (resp., (c)) can be intepreted as a local-global statement for the isomorphism classes of quadratic (resp., hermitian) forms. We also note that one has a result similar to part (c) for absolutely almost simple simply connected groups of type $\mathsf{C}_n$ that split over a quadratic extension --- see \cite[Remark 8.6]{CRR-Spinor}.)

\vskip2mm

Our next result relies only on the finiteness of unramified cohomology in degree 3.

\begin{thm}\label{T-Thm2}
Let $k$ be a number field and suppose $K = k(x_1, \dots, x_r)$ is a purely transcendental extension of $k$ or $K = k(X)$ is the function field of a Severi-Brauer $X$ variety over $k$ associated with a central division algebra $D$ over $k$ of degree $\ell.$ Let $V$ be any divisorial set of places of $K$.

\vskip2mm

\noindent {\rm (a)} \parbox[t]{16.5cm}{Let $m > 1$ be a square-free integer such that $k$ contains a primitive $m$-th root of unity. Assume that $m$ is either relatively prime to $\ell$ or $\ell$ is a prime number. If
$G = {\rm SL}_{1,A}$, with $A$ a central simple $K$-algebra of degree $m$, then the global-to-local map
$$
\lambda_{G,V} \colon H^1(K, G) \to \prod_{v \in V} H^1(K_v,G)
$$
is proper. In particular, $\Sha(G,V)$ is finite.}

\vskip2mm

\noindent {\rm (b)} \parbox[t]{16.5cm}{Let $G$ be a simple algebraic $K$-group of type $\mathsf{G}_2.$ Assume that either $\ell$ is odd or $\ell=2.$ Then the number of $K$-isomorphism classes of $K$-forms $G'$ of $G$ having good reduction at all $v \in V$ is finite, and, moreover, the global-to-local map
$$
\lambda_{G,V} \colon H^1(K, G) \to \prod_{v \in V} H^1(K_v,G)
$$
is proper. In particular, $\Sha(G,V)$ is finite. }

\end{thm}
\begin{proof}
In view of Theorem \ref{T-PurelyTran} and Proposition \ref{P-SB}, the statements are proved by imitating the proofs of Theorem 5.7 and 9.1 in \cite{CRR-Spinor}.
\end{proof}

\vskip2mm

\noindent {\bf Remark 5.7.} While we have formulated the theorem assuming that $k$ is a number field, using Remark 5.2 and Corollary \ref{C-SB}, one can extend the result to purely transcendental extensions of finite fields and function fields of Severi-Brauer varieties over global fields of positive characteristic.


\vskip2mm

We conclude this section with a simple proof of a (known) result on the Brauer group of  a Severi-Brauer variety. Initially, we were unable
to find a suitable reference in the literature and developed the argument given below; however, subsequently Skip Garibaldi pointed out to us that this fact is a particular case of Theorem B in \cite{MerTig}. Let $k$ be an arbitrary field, fix a separable closure $k^{\rm sep}$ of $k$, and suppose $X$ is a Severi-Brauer variety associated with a central division algebra $D$ over $k$. Then $\overline{X} = X \times_k k^{\rm sep}$ is isomorphic to $\mathbb{P}^n_{k^{\rm sep}}$ for some $n.$ Note that since $X$ is smooth, the (cohomological) Brauer group $\Br(X)$ of $X$ coincides with the unramified Brauer group $\Br(k(X))_{V_0}$ of the function field $L(X)$ with respect to the geometric places $V_0$ (see \cite[Proposition 2.1]{GrBr3}). 

\addtocounter{thm}{1}

\begin{prop}
The natural map $\Br(k) \to \Br(X)$ is surjective.
\end{prop}

\begin{proof}
Consider the Hochschild-Serre spectral sequence
$$
E_2^{p,q} = H^p (k, H^q_{\text{\'et}} (\overline{X}, \mathbb{G}_{m, \overline{X}})) \Rightarrow H^{p+q}_{\text{\'et}} (X, \mathbb{G}_{m, X}).
$$
Then the sequence of low-degree terms
$$
E_2^{2,0} \to \ker (E^2 \to E_2^{0,2}) \to E_2^{1,1}
$$
yields the exact sequence
$$
\Br (k) \to \Br_1 (X) \to H^1(k, {\rm Pic}(\overline{X})),
$$
where $\Br_1(X) = \ker (\Br(X) \to \Br(\overline{X}))$ is the (so-called) {\it algebraic} Brauer group. Since
$$\Br(\overline{X}) = \Br(\mathbb{P}_{k^{\rm sep}}^n) = \Br(k^{\rm sep}) = 0,$$ it follows that $\Br_1(X) = \Br(X).$ Moreover, ${\rm Pic}(\overline{X}) \simeq {\rm Pic}(\mathbb{P}_{k^{\rm sep}}^n) \simeq \Z$, so to complete the proof it suffices to show that ${\rm Pic}(\overline{X})$ has trivial $\Ga(k^{\rm sep}/k)$-action (as then $H^1(k, \Z) = 0$). Now $\Ga(k^{\rm sep}/k)$ can act on $\Z$ only by sending 1 to either 1 or -1, and we need to eliminate the second possibility.
For this, we interpret the action in terms of line
bundles and observe that the line bundles in the class corresponding to $1$ have nonzero global sections while those in the class corresponding to $-1$ do not,
so the required fact follows immediately.

\end{proof}

We recall that according to a theorem of Amitsur \cite{Amitsur}, the kernel of the natural map $\Br(k) \to \Br(X)$ coincides with the cyclic subgroup of $\Br(k)$ generated by the class of $D$; in particular, we obtain $\Br(X) \simeq \Br(k)/\langle [D] \rangle.$

\vskip3mm

\noindent {\small {\bf Acknowledgements.} We are grateful to Michael Rapoport for raising some interesting questions about tori with good reduction and to Skip Garibaldi for useful comments.}

\vskip5mm

\bibliographystyle{amsplain}

\end{document}